\documentclass{amsart}

\usepackage{todonotes}

\newcommand{\numberseries}{\bfseries}   

\newlength{\thmtopspace}                
\newlength{\thmbotspace}                
\newlength{\thmheadspace}               
\newlength{\thmindent}                  

\newtheoremstyle{fixed bf head,slanted body}
                {\thmtopspace}{\thmbotspace}{\slshape}
                {\thmindent}{\bfseries}{.}{\thmheadspace}
                {{\numberseries \thmnumber{#2\;}}\thmname{#1}\thmnote{ (#3)}}

\newtheoremstyle{variable bf head,slanted body}
                {\thmtopspace}{\thmbotspace}{\slshape}
                {\thmindent}{\bfseries}{.}{\thmheadspace}
                {{\numberseries \thmnumber{#2\;}}\thmname{#1}\thmnote{ #3}}

\newtheoremstyle{fixed bf head,upright body}
                {\thmtopspace}{\thmbotspace}{\upshape}
                {\thmindent}{\bfseries}{.}{\thmheadspace}
                {{\numberseries \thmnumber{#2\;}}\thmname{#1}\thmnote{ (#3)}}

\newtheoremstyle{numbered paragraph}
                {\thmtopspace}{\thmbotspace}{\upshape}
                {\thmindent}{\upshape}{}{\thmheadspace}
                {{\numberseries \thmnumber{#2.}}}

\usepackage{amsfonts,amsmath,amssymb,amsthm,amscd,amsxtra}
\usepackage{enumerate,verbatim}

\usepackage[all,2cell,ps]{xy}
\usepackage[pagebackref]{hyperref}
\usepackage{mathptmx}

\theoremstyle{plain} 

\newtheorem{thm}{Theorem}[section]

\newtheorem*{setup}{Setup}
\newtheorem{cor}[thm]{Corollary}

\theoremstyle{definition}

\newtheorem{eg}[thm]{Example}
\newtheorem{conj}[thm]{Conjecture}
\newtheorem{ques}[thm]{Question}

\newtheorem{chunk}[thm]{\hspace*{-1.065ex}\bf}

\newcommand{\fm}{\mathfrak{m}}
\newcommand{\fn}{\mathfrak{n}}
\newcommand{\fp}{\mathfrak{p}}

\newcommand{\CC}{\mathbb{C}}

\newcommand{\ZZ}{\mathbb{Z}}

\DeclareMathOperator{\CI}{\textnormal{CI-dim}}
\DeclareMathOperator{\G-dim}{\textnormal{G-dim}}

\DeclareMathOperator{\cx}{cx}
\DeclareMathOperator{\Tr}{Tr}

\DeclareMathOperator{\md}{\operatorname{\mathsf{mod}}}

\DeclareMathOperator{\depth}{depth}

\DeclareMathOperator{\Ext}{Ext}

\DeclareMathOperator{\Hom}{Hom}

\DeclareMathOperator{\pd}{pd}

\DeclareMathOperator{\Supp}{Supp}
\DeclareMathOperator{\Spec}{Spec}

\DeclareMathOperator{\Tor}{Tor}

\DeclareMathOperator{\cod}{codim}

\DeclareMathOperator{\length}{length}

\def\urltilda{\kern -.15em\lower .7ex\hbox{\~{}}\kern .04em}
\def\urldot{\kern -.10em.\kern -.10em}\def\urlhttp{http\kern -.10em\lower -.1ex
\hbox{:}\kern -.12em\lower 0ex\hbox{/}\kern -.18em\lower 0ex\hbox{/}}

\begin{document}

\title[Maximal Cohen-Macaulay tensor products]{Maximal Cohen-Macaulay tensor products}

\author[Olgur Celikbas]{Olgur Celikbas}
\address{Department of Mathematics \\
West Virginia University\\
Morgantown, WV 26506 U.S.A}
\email{olgur.celikbas@math.wvu.edu}

\author[Arash Sadeghi]{Arash Sadeghi}
\address{School of Mathematics, Institute for Research in Fundamental Sciences (IPM), P.O. Box:19395-5746, Tehran, Iran} \email{sadeghiarash61@gmail.com}

\thanks{Sadeghi's research was supported by a grant from the Institute for Research in Fundamental Sciences (IPM)}


\begin{abstract} In this paper we are concerned with the following question: if the tensor product of finitely generated modules $M$ and $N$ over a local complete intersection domain is maximal Cohen-Macaulay, then must $M$ or $N$ be a maximal Cohen-Macaulay? Celebrated results of Auslander, Lichtenbaum, and Huneke and Wiegand, yield affirmative answers to the question when the ring considered has codimension zero or one, but the question is very much open for complete intersection domains that have codimension at least two, even open for those that are one-dimensional, or isolated singularities. Our argument exploits Tor-rigidity and proves the following, which seems to give a new perspective to the aforementioned question: if $R$ is a complete intersection ring which is an isolated singularity such that $\dim(R)>\cod(R)$, and the tensor product $M\otimes_RN$ is maximal Cohen-Macaulay, then $M$ is maximal Cohen-Macaulay if and only if $N$ is maximal Cohen-Macaulay. 
\end{abstract}

\maketitle{}

\section{Introduction}

Throughout $R$ denotes a commutative Noetherian local ring with unique maximal ideal $\fm$ and residue field $k$. Furthermore $\md R$ denotes the category of all finitely generated $R$-modules. For the standard, unexplained notations and basic terminology, we refer the reader to \cite{AuBr, Av2, BH}.

It is well-known that the tensor product $M\otimes_RN$ of two nonzero modules $M, N \in \md R$ tends to have torsion in general, even if $M$ and $N$ are ideals, or maximal Cohen-Macaulay modules. Therefore one expects that the modules $M$ and $N$ should enjoy strong homological properties in the rare case that their tensor product has no torsion. The study of such situation was initiated by Auslander \cite{Au} and Lichtenbaum \cite{Li} over regular local rings, and continued by Huneke and Wiegand \cite{HW1} over hypersurface rings. Their results easily lead one to raise the following question, which has an affirmative answer when $c=0$ and $c=1$ by \cite[3.1]{Au} and \cite[2.7]{HW1} \footnote{ This applies since the ring in Question \ref{q1} is a domain; see the erratum \cite{HWE}.}, respectively; see also \cite{GORS2, GORS1, Bounds}. 

\begin{ques} \label{q1} Let $R$ be a local complete intersection domain of codimension $c$. If $M, N \in \md R$ are nonzero modules such that $M\otimes_RN$ satisfies \emph{Serre's condition} $(S_{c+1})$, then must $M$ or $N$ satisfy $(S_{c+1})$?
\end{ques}

Recall, in our context, a module $M\in \md R$ satisfies $(S_{n})$ if it is an $n$th syzygy module; see \cite[3.8]{EG}. In this paper we consider Question \ref{q1} for the case where the ring is an isolated singularity and the tensor product satisfies Serre's condition $(S_{n})$ for each nonnegative integer $n$. In other words, we consider the case where $R$ is a complete intersection domain such that $R_{\fp}$ is regular for each prime ideal $\fp \in \Spec(R)-\{\fm\}$, and $M\otimes_RN$ is maximal Cohen-Macaulay. To our best knowledge, Question \ref{q1}, even under these very strong hypotheses, is open. Motivated by this fact, we propose:

\begin{conj} \label{conj1} Let $R$ be a local complete intersection domain, and let $M, N \in \md R$. Assume $R$ is an isolated singularity. If $M\otimes_RN$ is maximal Cohen-Macaulay, then $M$ or $N$ is maximal Cohen-Macaulay.
\end{conj}

Conjecture \ref{conj1}, similar to Question \ref{q1}, holds when the codimension of the ring is at most one; see \cite[3.1]{Au} and \cite[3.1]{HW1}. It is worth pointing out, when the ring in question has dimension one, we have the following basic problem, which, quite interestingly, still remains open in general; see \cite[2.10]{Survey}.

\begin{conj} \label{conj2} Let $R$ be a one-dimensional local complete intersection domain, and let $M, N \in \md R$. If $M\otimes_RN$ is torsion-free, then $M$ or $N$ is torsion-free.
\end{conj}

Conjecture \ref{conj2} is an easy consequence of a conjecture of Dao \cite[4.2]{DaoComp} which, in particular, claims that all finitely generated modules over one-dimensional local complete intersection domains are Tor-rigid \cite{Au}. We should also mention that Conjectures \ref{conj1} and \ref{conj2} fail if the ring considered is not assumed to be a complete intersection. Constapel \cite{Constapel} constructed a one-dimensional Gorenstein local domain $R$ that is not a complete intersection, and $M, N\in \md R$ with $M\otimes_RN$ is torsion-free, but both $M$ and $N$ have torsion.

The main purpose of this paper is to make progress pertaining to Conjecture \ref{conj1}. Our argument yields:

\begin{thm} \label{thmintro} Let $R$ be a local complete intersection with $\dim(R)>\cod(R)$, and let $M, N \in \md R$. Assume $R$ is an isolated singularity. If $M\otimes_RN$ is maximal Cohen-Macaulay, then $M$ is maximal Cohen-Macaulay if and only if $N$ is maximal Cohen-Macaulay. In other words, $M\otimes_RN$ cannot be maximal Cohen-Macaulay in case $M$ is maximal Cohen-Macaulay but $N$ is not maximal Cohen-Macaulay.
\end{thm}

It is easy to see that the condition $\dim(R)>\cod(R)$ is necessary for Theorem  \ref{thmintro}: if $R=k[[x,y]]/(xy)$, $M=R/xR$ and $N=R/(x^2)R$, then $M\otimes_RN \cong M$ is maximal Cohen-Macaulay, but $N$ is not. 

In section 2, we state our main result, give a proof for Theorem \ref{thmintro} and discuss several consequences of our argument. Section 3 is devoted to various preliminary results and a proof of our main result; see Theorems \ref{l4intro} and  \ref{l4}.

\section{Statement of the Main Result and corollaries}

In the following, for a given module $M\in \md R$, $\G-dim_R(M)$ and $\CI_R(M)$ denote the \emph{Gorenstein} dimension \cite{AuBr} and the \emph{complete intersection} dimension \cite{AGP} of $M$, respectively. Moreover, $\Tr M$ denotes the Auslander \emph{transpose} of $M$; see \cite{AuBr}. 

The next theorem is our main result. In this section we discuss some of its consequences and defer its proof to section 3; see Theorem \ref{l4}.

\begin{thm}\label{l4intro} Let $R$ be a local complete intersection ring  and let $M,N \in \md R$. Assume:
\begin{enumerate}[\rm(i)]
\item $\depth_R(M\otimes_RN)\geq\depth_R(N)$.
\item $\length(\Tor_i^R(M,N))<\infty$ for all $i\gg0$.
\item $M_\fp$ is maximal Cohen-Macaulay for all $\fp \in \Supp_R(M\otimes_RN)-\{\fm\}$.
\item If $\depth_R(N)\leq \cod(R)$, assume further $\Tor_i^R(M,N)=0$ for all $i=1, \ldots, \cod(R)-\depth_R(N)+1$.
\end{enumerate}
Then $\Tor_i^R(M,N)=0$ for all $i\geq 1$.
\end{thm}

Theorem \ref{thmintro}, stated in the introduction, is a special case of the following consequence of Theorem \ref{l4intro}. 

\begin{cor}\label{th6} Let $R$ be a local complete intersection ring and let $M, N \in \md R$. Assume:
\begin{enumerate}[(a)]
\item $M\otimes_RN$ is maximal Cohen-Macaulay.
\item $\length(\Tor_i^R(M,N))<\infty$ for all $i\gg 0$. 
\item If $\dim(R) \leq \cod(R)$, assume further  $\Tor_i^R(M,N)=0$ for all $i=1, \ldots, \cod(R)-\dim(R)+1$.
\end{enumerate}
Then the following conditions are equivalent:
\begin{enumerate}[(i)]
\item $M$ is maximal Cohen-Macaulay.
\item $N$ is maximal Cohen-Macaulay.
\item $\Tor_i^R(M,N)=0$ for all $i\geq 1$.
\item $\Ext^i_R(M,N)=0$ for all $i\geq 1$.
\item $\Ext^i_R(N,M)=0$ for all $i\geq 1$. 
\end{enumerate}
\end{cor}

Prior giving a proof of Corollary \ref{th6}, we record:

\begin{chunk} (\cite[3.5]{ArY}) \label{df} If $R$ is a local ring and $M,N\in \md R$ are nonzero such that $\Tor_i^R(M,N)=0$ for all $i\geq 1$ and $\CI_R(M)<\infty$, then $\depth_{R}(M)+\depth_{R}(N)=\depth(R)+\depth_{R}(M\otimes_{R}N)$.
\end{chunk}

\begin{chunk} \label{len} If $R$ is a local ring and $M,N\in \md R$ such that $\CI_{R_{\fp}}(M_{\fp})=0$ for each $\fp\in\Spec R-\{\fm\}$, then $\length(\Tor_i^R(M,N))<\infty$ for all $i\gg 0$ if and only if $\length(\Tor_i^R(M,N))<\infty$ for all $i\geq 1$; see \cite[4.9]{AvBu}. 
\end{chunk}

\begin{proof}[A proof of Corollary \ref{th6}] It suffices to assume (i) and prove (ii) and (iii) hold; see \cite[4.8, 6.1 and 6.3]{AvBu}.

Assume (i). Then we have 
$\depth_{R_\fp}(M_\fp)+\depth_{R_\fp}(N_\fp)=\depth R_\fp+\depth_{R_\fp}(M_\fp\otimes_{R_\fp}N_\fp)$; see \ref{df} and \ref{len}. Consequently $N_\fp$ is maximal Cohen-Macaulay for all $\fp\in\Supp(M)\cap\Supp(N)-\{\fm\}$.
Now, by swaping the role of $N$ with that of $M$, we see from Theorem \ref{l4intro} that $\Tor_i^R(M,N)=0$ for all $i\geq 1$. Thus \ref{df} shows that $N$ is maximal Cohen-Macaulay.
\end{proof}

The next two corollaries are reformulations of Corollary \ref{th6}; concerning the first one, see also \cite[1.2]{Jo2}.

\begin{cor} \label{cor2} Let $R$ be a local complete intersection. Assume $\dim(R)>\cod(R)$ and $R$ is an isolated singularity. If $M\in \md R$ with $\pd_R(M)=\infty$, then $\Omega^i_R(M) \otimes_R \Omega^j_R(M)$ is not maximal Cohen-Macaulay for each $i\geq 0$ and each $j\geq \dim(R)$.
\end{cor}

\begin{cor}  \label{cor3} Let $R$ be a local complete intersection with $\dim(R)>\cod(R)$. If $J\subseteq I$ are ideals of $R$ such that $\length(\Tor_i^R(R/I,R/J))<\infty$ for all $i\gg0$, and $R/I$ is maximal Cohen-Macaulay, then $R/J$ is maximal Cohen-Macaulay. In particular, if $M$ is a maximal Cohen-Macaulay $R$-module such that $M_{\fp}$ is free over $R_{\fp}$ for all $\fp \in \Spec(R)-\{\fm\}$, then $R/J$ is maximal Cohen-Macaulay for all ideals $J$ of $R$ with $JM=0$.
\end{cor}

The example mentioned in the introduction also shows that the hypothesis $\dim(R)>\cod(R)$ in Corollary \ref{cor2} and \ref{cor3} is necessary: if $R=k[[x,y]]/(xy)$, $M=R/I$ and $N=R/J$, where $I=xR$ and $J=(x^2)R$, then $\pd_R(M)=\infty$ but $\Omega^1_R(M)\otimes_R \Omega^1_R(M)$ is maximal Cohen-Macaulay.

We can also show that the hypothesis concerning the dimension and codimension in Corollary \ref{th6}  is necessary. Let us first recall two elementary facts:

\begin{chunk} (\cite[2.8]{AuBr}\label{a1}) Let $M\in \md R$ and let $n\geq0$ be an integer. Then there is an exact sequences of functors:
	$$0 \longrightarrow \Ext^1_R(\Tr\Omega^nM,-)\longrightarrow\Tor_n^R(M,-)\longrightarrow\Hom_R(\Ext^n_R(M,R),-)
	\longrightarrow\Ext^2_R(\Tr\Omega^{n+1}M,-).$$
\end{chunk}

When $M\in \md R$ and $\G-dim_R(M)<\infty$, we denote by $\widehat{\Ext}^n_R(M,N)$ (respectively, by $\widehat{\Tor}^n_R(M,N)$) the $n$th \emph{Tate cohomology} (respectively, the $n$th \emph{Tate homology}) of $M$ and $N$; see, for example, \cite{AM}. 

\begin{chunk} (\cite[4.7]{Ce}) \label{OC} If $R$ is a a local complete intersection ring with $\cod(R)=\dim(R)\geq 1$, and $M$ and $N$ are maximal Cohen-Macaulay modules such that $\length(M\otimes_RN)<\infty$, then $\Tor_i^R(M,N)\neq 0$ if and only if $i$ is a nonnegative even integer.
\end{chunk}

\begin{eg}\label{exa} Let $R=\CC[\![x_1, x_2, y_1, y_2]\!]/(x_1y_1, x_2y_2)$ be the formal power series ring, $M=R/(x_1,x_2)$ and $N=R/(y_1, y_2)$. Then $R$ is a complete intersection with $\cod(R)=\dim(R)=2$. Notice $M$ and $N$ are maximal Cohen-Macaulay with $M\cong M^{\ast}$, $M\otimes_RN\cong k$ and $\widehat{\Ext}^0_R(M,N)=\widehat{\Ext}^1_R(M,N)= 0$; see \cite[3.17]{BJVH}. Consider a maximal Cohen-Macaulay $R$-module $L$ with $M\cong \Omega^1_R(L)$. Then, for each $i\geq 1$, one has:
\[\begin{array}{rl}\notag{}
\Ext^i_R(\Tr L,N)\cong\widehat{\Ext}^i_R(\Tr L,N)
\cong\widehat{\Ext}^{i-2}_R(L^*,N)
\cong\widehat{\Ext}^{i-1}_R(M,N).
\end{array}\]
In particular, we have $\Ext^1_R(\Tr L,N)=\Ext^2_R(\Tr L,N)=0$. This implies $L\otimes_R N \cong \Hom_R(L^{\ast},N)$ and hence $L\otimes_RN$ is maximal Cohen-Macaulay; see \ref{a1} and \cite[1.4.18]{BH}.

We have, since $\Tor_i^R(M,N)\neq 0$ if and only if $i$ is a nonnegative even integer, that 
$\Tor_{2i}^R(L,N)=0$ and $\Tor_{2i+1}^R(L,N)\neq 0$ for all $i\geq 1$; see \ref{OC}. Furthermore, as $M\cong \Omega^1_R(L)$ and $L\otimes_RN$ is maximal Cohen-Macaulay, it is easy to see, using the depth lemma, that $\Tor_1^R(L,N)\neq 0$.

Consequently $(L,N)$ is a pair of maximal Cohen-Macaulay $R$-modules such that $L\otimes_RN$ is maximal Cohen-Macaulay, $\length(\Tor_i^R(L,N))<\infty$ for all $i\geq 1$ and $\Tor_i^R(L,N)=0$ if and only if $i$ is a nonnegative even integer; cf. Corollary \ref{th6}.
\end{eg}

\section{A proof of the main result}

This section is devoted to a proof of Theorem \ref{l4intro}, which requires some preparation. 
We skip the proof of the next elementary remark as one can prove it easily following the argument of \cite[2.2]{CSY}.

\begin{chunk}\label{p1} Let $M,N \in \md R$ and let $n\geq 1$ be an integer. 
\begin{enumerate}[\rm(i)]
\item Assume $n\geq3$. If $\depth_R(\Hom_R(M,N))\geq n$, $\depth_R(N)\geq n-1$ and $\length(\Ext^i_R(M,N))<\infty$ for all $i=1, \ldots, n-2$, then $\Ext^i_R(M,N)=0$ for all $i=1, \ldots, n-2$.
\item If $\depth_R(M\otimes_RN)\geq n$, $\depth_R(N)\geq n-1$ and $\length(\Ext^i_R(\Tr M,N))<\infty$ for all $i=1, \ldots, n$, then $\Ext^i_R(\Tr M,N)=0$ for all $i=1, \ldots, n$.
\item If $\Ext^i_R(\Tr M,N)=0$ for all $i=1, \ldots, n$, then 
$\depth_R(M\otimes_RN)\geq\min\{n,\depth_R(N)\}$.
\end{enumerate}	
\end{chunk}

\begin{chunk} (\cite[4.7]{AvBu} and \cite[4.1]{SS}) \label{l6}
Let $M, N \in \md R$ be modules. Assume $\G-dim_R(M)<\infty$. Assume further $\CI_R(M)<\infty$ or $\CI_R(N)<\infty$. Then one has:
\begin{equation}\notag{}
\widehat{\Ext}^i_R(M,N)=0 \text{ for all } i\gg0 \Longleftrightarrow \widehat{\Ext}^i_R(M,N)=0 \text{ for all } i \ll 0 \Longleftrightarrow  \widehat{\Ext}^i_R(M,N)=0 \text{ for all } i\in\ZZ.
\end{equation}
\end{chunk}

\begin{chunk}\label{le1} Let $R$ be a local ring and let $M, N \in \md R$. Assume $\CI_R(N)<\infty$. Then,
\begin{enumerate}[\rm(i)]
\item If $\G-dim_R(M)<\infty$ and $\Ext^i_R(M,N)=0$ for all $i\gg0$, then $\G-dim_R(M)=\sup\{i\in \ZZ \mid\Ext^i_R(M,N)\neq0\}$.
\item If $M$ is totally reflexive, then $\Ext^i_R(\Tr M,N)=0$ for all $i\geq 1$ if and only if $\Tor_i^R(M,N)=0$ for all $i\geq 1$.
\end{enumerate}
\end{chunk}

\begin{proof} Part (i) follows from \ref{l6} and \cite[4.3]{SS}. Therefore we proceed to prove part (ii).
Note that $\Tr M$ is totally reflexive. Hence, by part (i) and \ref{l6}, we have:
\begin{equation}\tag{\ref{le1}.1}
\Ext^i_R(\Tr M,N)=0 \text{ for all } i\geq 1 \Longleftrightarrow \widehat{\Ext}^i_R(\Tr M,N)=0 \text{ for all  } i\in\ZZ.
\end{equation}
Since $M^*\cong \Omega^2\Tr M$ up to projectives, we can make use of \cite[4.4.7]{AvBu} and obtain:
\begin{equation}\tag{\ref{le1}.2}
\widehat{\Tor}_i^R(M,N)=0 \text{ for all } i\in\ZZ \Longleftrightarrow \widehat{\Ext}^i_R(\Tr M,N)=0 \text{ for all  } i\in\ZZ. 
\end{equation}
On the other hand, by \cite[4.9]{AvBu}, we know:
\begin{equation}\tag{\ref{le1}.3}
\widehat{\Tor}_i^R(M,N)=0 \text{ for all  } i\in\ZZ \Longleftrightarrow \Tor_i^R(M,N)=0 \text{ for all }  i\geq 1.
\end{equation}
Now the assertion follows from (\ref{le1}.1), (\ref{le1}.2) and (\ref{le1}.3).
\end{proof}

In passing we record a consequence of \ref{p1} and \ref{le1}, which can also be used to see the vanishing of $\Tor_{2i}^R(L,N)$ in Example \ref{exa}; cf. \ref{OC}.

\begin{chunk}\label{p2} If $R$ is a a local complete intersection ring with $\cod(R)=\dim(R)$, and $M$ and $N$ are maximal Cohen-Macaulay modules such that $M\otimes_RN$ is maximal Cohen-Macaulay and $\length(\Tor_i^R(M,N))<\infty$ for $i\gg0$, then $\Tor_{2i}^R(M,N)=0$ for all $i\geq 1$.
\end{chunk}

\begin{proof} Set $d=\cod(R)=\dim(R)$. Then it follows from \ref{len}, \ref{p1} and \ref{le1} that $\Ext^i_R(\Tr M,N)=0$ for all $i=1, \ldots, d$. Therefore $\widehat{\Tor}_i^R(M,N) \cong \widehat{\Ext}^{-i+1}_R(\Tr M,N)=0$ for all $i=-d+1, \ldots, 0$; see \cite[4.4.7]{AvBu}. Now the assertion follows from \cite[3.9]{BJVH}.
\end{proof}

From now on, we assume the following setup throughout the rest of the paper.

\begin{setup} Assume $R=S/(\textbf{x})$, where $(S, \fn)$ is a Gorenstein local ring and $\textbf{x}=x_1, \ldots, x_c$ is an $S$-regular sequence in $\fn$ for some $c\geq 0$. Assume further $M, N\in \md R$ are nonzero modules such that $\CI_S(N)<\infty$ and $N$ is Tor-rigid over $S$. In particular we have that $\CI_R(N)<\infty$; see \cite[1.12.2]{AGP}.
\end{setup}

The following is a generalization of \cite[4.1]{Sa1}, which plays an important role in the proof of our main result, Theorem \ref{l4intro}.

\begin{chunk}\label{th9} \label{cc1} Assume $n\geq 1$ and $v\geq 0$ are integers, and that the following conditions hold.
\begin{enumerate}[\rm(i)]
\item $M$ satisfies $(S_v)$.
\item $\depth_R(N)\leq n+c+v$.
\item $\Ext^{i}_R(M,N)=0$ for all $i=n, \ldots, n+c$.
\end{enumerate}
Then $\G-dim_R(M)=\sup\{i\in \ZZ \mid\Ext^i_R(M,N)\neq0\}\leq n-1$.
\end{chunk}

\begin{proof} We proceed by induction on $v$. Assume first the case where $v=0$ holds, and suppose $v\geq 1$. Then consider the pushforward of $M$, i.e., an exact sequence in $\md R$ of the form
\begin{equation}\tag{\ref{th9}.1}
0\rightarrow M\rightarrow F\rightarrow M_1\rightarrow0,
\end{equation}
where $F$ is free; see, for example \cite[page 174]{Ce}. It follows from (\ref{th9}.1) that $M_1$ satisfies $(S_{v-1})$ and 
\begin{equation}\tag{\ref{th9}.2}
\Ext^i_R(M,N)\cong\Ext^{i+1}_R(M_1,N) \text{ for all } i\geq 1.
\end{equation}
Therefore $\Ext^i_R(M_1,N)=0$ for all $i=n+1, \ldots, n+c+1$. So, by the induction hypothesis, we conclude that $\Ext^i_R(M_1,N)=0$ for all $i\geq n+1$. Now, in view of (\ref{th9}.2), the claim follows from \ref{l6} and \ref{le1}(i).	

Next we proceed by induction on $c$ to establish the case where $v=0$. 

If $c=0$, then $N$ is a Tor-rigid $R$-module and the assertion follows from \cite[5.8(2)]{CGZS} and \cite[2.5]{ArY}. Now let $c\geq 1$ and  set $Q=S/(x_1,\ldots, x_{c-1})$. Therefore $R\cong Q/(x_c)$. 
It follows from  \cite[11.66]{Rotman} that there is a long exact sequence of the form:
\begin{equation}\tag{\ref{th9}.3}
\cdots\rightarrow\Ext^i_R(M,N)\rightarrow\Ext^i_Q(M,N)\rightarrow\Ext^{i-1}_R(M,N)\rightarrow\Ext^{i+1}_R(M,N)\rightarrow\cdots.
\end{equation}
It follows that $\Ext^i_Q(M,N)=0$ for all $i=n+1, \ldots, n+c$, and so, by the induction hypothesis, we conclude that $\G-dim_Q(M)=\sup\{i\mid\Ext^i_Q(M,N)\neq0\}\leq n$. Thus, by ({\ref{th9}.3}), we have $\Ext^{i-1}_R(M,N)\cong\Ext^{i+1}_R(M,N)$ for all $i\geq n+1$. Since $c\geq 1$, it is clear that $\Ext^i_R(M,N)=0$ for all $i\geq n$. Consequently the claim follows from  \ref{l6} and \ref{le1}(i).
\end{proof}

\begin{chunk}\label{l3} Set $t=\depth_R(N)$. Assume $t\geq c+1$ and $\Ext^i_R(M,N)=0$ for all $i=1, \ldots, t$. Then $M$ is maximal Cohen-Macaulay, and $\Ext^i_R(M,N)=0$ for all $i\geq 1$.
\end{chunk}

\begin{proof} If $t=c+1$, then the assertion follows from \ref{th9}. Hence we assume $t\geq c+2$.
Pick a regular sequence $y_1,\ldots,y_{t-c-1}$ on $N$, and set $N_j=N/(y_1,\ldots,y_j)N$ for each $j=0,1, \ldots, t-c-1$, where $N_0=N$. This yields, for each $j=0, 1, \ldots, t-c-1$, a short exact sequence of the form:
\begin{equation} \tag{\ref{l3}.1}
0 \to N_j \overset{y_{j+1}}{\longrightarrow} N_j \longrightarrow N_{j+1}\longrightarrow 0.
\end{equation}
Since $N$ is Tor-rigid over $S$, it follows from (\ref{l3}.1) and Nakayama's lemma that $N_{j}$ is  Tor-rigid over $S$ for each $j=1, \ldots, t-c-1$.
Moreover, for each $j=0, 1, \ldots, t-c-1$, (\ref{l3}.1) induces a long exact sequence of the form: 
\begin{equation} \tag{\ref{l3}.2}
\cdots\rightarrow\Ext^i_R(M,N_j)\rightarrow\Ext^i_R(M,N_{j+1})\rightarrow\Ext^{i+1}_R(M,N_j)\rightarrow\cdots.\end{equation}
Setting $j=0$ and using our assumption, we see from  (\ref{l3}.2) that $\Ext^i_R(M,N_1)=0$ for all $i=1, \ldots, t-1$. By repeating this argument, we conclude: 
$$\Ext^i_R(M,N_{t-c-1})=0 \text { for all } i=1, \ldots, c+1.$$ 
Notice $\depth_R(N_{t-c-1})=c+1$. Therefore, by using \ref{th9} for the pair $(M, N_{t-c-1})$ and for the cases where $n=1$ and $v=0$, we conclude $\G-dim_R(M)=\sup\{i\mid\Ext^i_R(M,N_{t-c-1})\neq0\}\leq 0$. Thus, since $M$ is not zero, we deduce that $M$ is maximal Cohen-Macaulay and $\Ext^i_R(M,N_{t-c-1})=0$ for all $i\geq 1$. Now, proceding inductively and using Nakayama's lemma, we see from (\ref{l3}.2) that $\Ext^i_R(M,N)$ vanishes for all $i\geq 1$.
\end{proof}

We are now ready to give a proof of our main result. Note that, since modules over regular local rings are Tor-rigid \cite{Au}, the conditions in our setup hold when $S$ is regular, i.e., when $R$ is a complete intersection ring. In particular the result below contains Theorem \ref{l4intro}.

\begin{thm}\label{l4} Set $t=\depth_R(N)$ and assume the following conditions hold:
\begin{enumerate}[(i)]
\item $\depth_R(M\otimes_RN)\geq t$.
\item $\length(\Tor_i^R(M,N))<\infty$ for all $i\gg0$.
\item $M_\fp$ is maximal Cohen-Macaulay for all $\fp\in\Supp_R(M\otimes_RN)-\{\fm\}$.
\item If $t\leq c$, assume further that $\Tor_i^R(M,N)=0$ for all $i=1, \ldots, c-t+1$.
\end{enumerate}
Then $\Tor_i^R(M,N)=0$ for all $i\geq 1$.
\end{thm}

\begin{proof} We start by recalling that $\CI_R(N)<\infty$. It follows from \ref{le1} that $\length(\Tor_i^R(M,N))<\infty$ and $\length(\Ext^i_R(\Tr M,N))<\infty$, for all $i\geq 1$. Hence, by \ref{p1}(ii), we have: 
\begin{equation}\tag{\ref{l4}.1}
\Ext^i_R(\Tr M,N)=0 \text{ for all } i=1, \ldots, t.
\end{equation}

Suppose $t\geq c+1$. Then (\ref{l4}.1) and \ref{l3} imply $\Tr M$ is maximal Cohen-Macaulay and $\Ext^i_R(\Tr M,N)=0$ for all $i\geq 1$. Thus $M$ is totally reflexive, and the vanishing of $\Tor_i^R(M,N)$ for all $i\geq 1$ follows from \ref{le1}(ii). 

Next assume $t\leq c$. Then, by assumption, we have $\Tor_i^R(M,N)=0$ for all $i=1, \ldots, n$, where $n=c-t+1$. We will proceed by induction on $n$.

First suppose $n=1$. We have, by \cite[2.21]{AuBr}, a short exact sequence of modules in $\md R$ of the form
\begin{equation}\tag{\ref{l4}.2}
0\longrightarrow F\longrightarrow L\longrightarrow M\longrightarrow0,
\end{equation}
where $F$ is a free $R$-module and $\Ext^1_R(L,R)=0$. Since $\Tor_1^R(M,N)=0$, tensoring (\ref{l4}.2) with $N$, we obtain the exact sequence:
\begin{equation}\tag{\ref{l4}.3}
0\longrightarrow N\otimes_RF\longrightarrow N\otimes_RL\longrightarrow N\otimes_RM\longrightarrow0.
\end{equation}
It follows from (\ref{l4}.2) that $\Tor_1^R(L,N)=\Tor_1^R(M,N)=0$ and $\Tor_i^R(L,N)=\Tor_i^R(M,N)$ for all $i\geq 2$. In particular, by (\ref{a1}), we see: 
\begin{equation}\tag{\ref{l4}.4}
\Ext^1_R(\Tr\Omega L,N)=0.
\end{equation}
It follows from ({\ref{l4}.2}) that $L_\fp$ is totally reflexive for all $\fp\in\Supp_R(L\otimes_RN)-\{\fm\}$. Therefore, since $\length(\Tor_i^R(L,N))<\infty$ for all $i\geq 1$, we observe from \ref{le1}(ii) that $\length(\Ext^i_R(\Tr L,N))<\infty$ for all $i\geq 1$. It is easy to see $\Tr L$ is stably isomorphic to $\Omega\Tr\Omega L$ because $\Ext^1_R(L,R)=0$. Furthermore, by (\ref{l4}.3), we have $\depth_R(L\otimes_RN)\geq \depth_R(N)$.  Thus \ref{p1} gives:
\begin{equation}\tag{\ref{l4}.5}
\Ext^{i+1}_R(\Tr\Omega L, N) \cong \Ext^{i}_R(\Omega\Tr\Omega L, N) \cong \Ext^{i}_R(\Tr L, N)=0 \text{ for all } i=1, \ldots t.
\end{equation}
Consequently, by (\ref{l4}.4) and (\ref{l4}.5), we obtain:
\begin{equation}\tag{\ref{l4}.6}
\Ext^i_R(\Tr\Omega L,N)=0 \text{ for all } i=1, \ldots, t+1. 
\end{equation}
We are now in a position to conclude that $\Tor_i^R(N,M)=0$ for all $i\geq 1$ by using \ref{th9} and \ref{le1}(ii).

Now we consider the case where $n\geq 2$. Set $T=S/(x_1,\cdots,x_{c-1})$ so that $R\cong T/(x_{c})$.
Consider an exact sequence of modules in $\md T$ of the form
\begin{equation}\tag{\ref{l4}.7}
0\rightarrow X\rightarrow P\rightarrow M\rightarrow0,
\end{equation}
where $P$ is a free $T$-module. Tensoring (\ref{l4}.7) with $N$ over $T$, we have the following exact sequence
\begin{equation}\tag{\ref{l4}.8}
0\rightarrow\Tor_1^T(M,N)\rightarrow X\otimes_TN\rightarrow P\otimes_TN\rightarrow M\otimes_TN\rightarrow0.
\end{equation}

The change of rings spectral sequence \cite[11.64]{Rotman} yields a long exact sequence of the form
$$\begin{CD}
&&&&&&&&\\
     &&&& \vdots&&\vdots&&\vdots\\
  \ \  &&&& \Tor_1^R(M,N)@>>>\Tor_2^T(M,N)@>>>\Tor_2^R(M,N)@>>>&\\
  \ \  &&&& M\otimes_RN@>>>\Tor_1^T(M,N)@>>>\Tor_1^R(M,N)@>>>0.&\\
\end{CD}$$\\
Since $\Tor_i^R(M,N)=0$ for all $i=1, \ldots, n$, it follows $M\otimes_RN \cong \Tor_1^T(M,N)$ and that $\Tor_i^T(M,N)=0$ for all $i=2, \ldots, n$. Hence 
$\Tor_i^T(X,N)=0$ for all $i=1, \ldots, n-1$. Similarly, it is easy to check that the hypotheses (i), (ii) and (iii) hold for the pair $(X,N)$ over the ring $T$. Therefore, by the induction hypothesis, we have $\Tor_i^T(X,N)=0$ for all $i\geq 1$, i.e., $\Tor_i^T(M,N)=0$ for all $i\geq 2$. This implies $\Tor_i^R(M,N)\cong\Tor_{i+2}^R(M,N)$ for all $i\geq 1$. Since, by assumption, $\Tor_1^R(M,N)$ and $\Tor_2^R(M,N)$ vanish, the claim follows.
\end{proof}

We finish this section by recording a consequence of Theorem \ref{l4}; cf. Corollary \ref{th6}. In the following $\cx_R(M)$ denotes the \emph{complexity} of $M$; see \cite{Av1}.

\begin{chunk}\label{th10} Assume, in our setup, the ring $(S, \fn)$ is an unramified or equi-characteristic regular local ring, and the regular sequence $\textbf{x}=x_1, \ldots, x_c$ is contained in $\fn^2$. Assume further the following conditions hold:
\begin{enumerate}[\rm(a)]
\item $M\otimes_RN$ is maximal Cohen-Macaulay.
\item $\length(\Tor_i^R(M,N))<\infty$ for all $i\gg0$.
\item $\cx_R(M)\neq c$ and $\cx_R(N)\neq c$.
\item If $\dim(R)<\cod(R)$, assume further $\Tor_i^R(M,N)=0$ for all $i=1, \ldots, \cod(R)-\dim(R)$.
\end{enumerate}
Then the following conditions are equivalent:
\begin{enumerate}[(i)]
\item $M$ is maximal Cohen-Macaulay.
\item $N$ is maximal Cohen-Macaulay. 
\item $\Tor_i^R(M,N)=0$ for all $i\geq 1$.
\item $\Ext^i_R(M,N)=0$ for all $i\geq 1$.
\item $\Ext^i_R(N,M)=0$ for all $i\geq 1$.
\end{enumerate}
\end{chunk}

\begin{proof} Proceeding as in the proof of Corollary \ref{th6}, it suffices to assume (i) and prove (ii) and (iii) hold; see \cite[4.8, 6.1 and 6.3]{AvBu}. Hence we assume (i) and induct on $c$. If $c=1$, then $M$ is free so the result follows. Thus suppose $c\geq 2$. 

Notice, for an indeterminate $z$ over $S$, the regular local ring $S[z]_{\fn S[z]}$ is unramified or equi-characteristic. Hence extending the residue field by using the faithfully flat extension $S\hookrightarrow S[z]_{\fn S[z]}$, we may assume $R$ is complete and has infinite residue field.

It follows there exists a regular sequence $\textbf{f} = f_1, f_2, \cdots , f_c$ on $S$ such that $(\textbf{f})S=(\textbf{x})S \subseteq \fn^2$, $R = R_1/(f)$ and $\cx_{R_1}(M)<\cod(R_1)=c - 1$, where $R_1 = S/(f_2, \cdots, f_c)$ and $f = f_1$; see \cite[1.3]{Jor} and also \cite[2.5]{Ce}. Repeating this argument, we obtain a regular sequence $\textbf{y}= y_1, y_2, \ldots , y_c$ on $S$ such that $(\textbf{y})S=(\textbf{f})S$, $R = R_{c-1}/(y_1,\ldots,y_{c-1})$ and $\cx_{R_{c-1}}(M)<\cod(R_{c-1})=1$, where $R_{c-1} = S/(y_c)$. Since $\pd_{R_{c-1}}(M)<\infty$, we have that $M$ is Tor-rigid as an $R_{c-1}$-module; see \cite[1.9]{HW2} and \cite[Theorem 3]{Li}. Therefore, for each $\fp\in\Supp(M\otimes_RN)-\{\fm\}$, it follows that
$\depth_{R_\fp}(M_{\fp})+\depth_{R_\fp}(N_\fp) = \depth R_\fp +\depth_{R_\fp}(M\fp\otimes_{R_\fp}N_\fp)$ and $N_\fp$ is maximal Cohen-Macaulay. Consequently, by swaping the role of $N$ with that of $M$, we conclude from Theorem \ref{l4} that $\Tor^R_i(M,N)=0$ for all $i\geq 1$.
\end{proof} 

\section*{Acknowledgments}
We would like to thank Roger Wiegand for useful discussions related to this work.


\begin{thebibliography}{9}
\bibitem{ArY}
Tokuji Araya and Yuji Yoshino. Remarks on a depth formula, a grade inequality and a conjecture of Auslander. \emph{Comm. Algebra},
26(11):3793--3806, 1998.

\bibitem{Au}
Maurice Auslander. Modules over unramified regular local rings. \emph{Illinois J. Math.}, 5:631--647, 1961. 

\bibitem{AuBr}
Maurice Auslander and Mark Bridger. \emph{Stable module theory}. Memoirs of the American Mathematical Society, No. 94. American
Mathematical Society, Providence, R.I., 1969. 

\bibitem{Av1}
Luchezar L. Avramov. Modules of finite virtual projective dimension. \emph{Invent. Math.}, 96(1):71--101, 1989. 

\bibitem{Av2}
Luchezar L. Avramov. Infinite free resolutions. In \emph{six lectures on commutative algebra (Bellaterra, 1996),} volume 166 of \emph{Progr.
Math.}, pages 1--118. Birkh\"{a}user, Basel, 1998.

\bibitem{AvBu}
Luchezar L. Avramov and Ragnar-O. Buchweitz. Support varieties and cohomology over complete intersections. \emph{Invent. Math.},
142(2):285--318, 2000.

\bibitem{AGP} 
Luchezar L. Avramov, Vesselin Gasharov, and Irena Peeva. Complete intersection dimension. \emph{Inst. Hautes \'{E}tudes Sci. Publ. Math.}, (86):67--114 (1998), 1997.

\bibitem{AM}
Luchezar L. Avramov and Alex Martsinkovsky. Absolute, relative, and Tate cohomology of modules of finite Gorenstein dimension.\emph{
Proceedings of the London Mathematical Society,} 85(2):393--440, 2002.

\bibitem{BJVH}
Petter Andreas Bergh and David A. Jorgensen. On the vanishing of homology for modules of finite complete intersection dimension.
\emph{J. Pure Appl. Algebra}, 215(3):242--252, 2011.

\bibitem{BH}
Winfried Bruns and J\"{u}rgen H. Herzog. \emph{Cohen-Macaulay rings}, volume 39 of \emph{Cambridge Studies in Advanced Mathematics.}
Cambridge University Press, Cambridge, 1993.

\bibitem{Ce}
Olgur Celikbas. Vanishing of Tor over complete intersections. \emph{J. Comm. Alg.}, Volume 3, Number 2, 2011. 

\bibitem{CGZS}
Olgur Celikbas, Mohsen Gheibi, Majid R. Zargar, and Arash Sadeghi. Homological dimensions of rigid modules. \emph{to appear in Kyoto Journal of Mathematics; posted at arXiv:1405.5188.}

\bibitem{GORS2}
Olgur Celikbas, Srikanth Iyengar, Greg Piepmeyer, and Roger Wiegand. Criteria for vanishing of Tor over complete intersections. \emph{Pacific Journal of Mathematics}, 276(1):93--115, 2015. 

\bibitem{GORS1}
Olgur Celikbas, Srikanth Iyengar, Greg Piepmeyer, and Roger Wiegand. Torsion in tensor powers of modules. \emph{Nagoya Mathematical
Journal,} 219:113--125, 2015.

\bibitem{Bounds}
Olgur Celikbas, Arash Sadeghi, and Ryo Takahashi. Bounds on depth of tensor products of modules. \emph{Journal of Pure and Applied
Algebra}, 219(5):1670--1684, 2015.

\bibitem{CSY}
Olgur Celikbas, Arash Sadeghi, and Yongwei Yao. Tensoring with the Frobenius endomorphism. \emph{to appear in Homology, Homotopy
and Applications; posted at arXiv: 1706.00238.} 

\bibitem{Survey}
Olgur Celikbas and Roger Wiegand. Vanishing of Tor, and why we care about it. \emph{Journal of Pure and Applied Algebra}, 219(3):429--448, 2015.

\bibitem{Constapel}
Petra Constapel. \emph{Length of Tor and torsion in tensor products.} ProQuest LLC, Ann Arbor, MI, 1995. Thesis (Ph.D.)-Northwestern University. 

\bibitem{DaoComp}
Hailong Dao. Picard groups of punctured spectra of dimension three local hypersurfaces are torsion-free. \emph{Compositio Mathematica}, 148(1):145--152, 2012. 

\bibitem{EG}
E. Graham Evans and Phillip Griffith. \emph{Syzygies}, volume 106 of \emph{London Mathematical Society Lecture Note Series}. Cambridge
University Press, Cambridge, 1985.

\bibitem{HW1}
Craig Huneke and Roger Wiegand. Tensor products of modules and the rigidity of Tor. \emph{Math. Ann.}, 299(3):449--476, 1994. 

\bibitem{HW2}
Craig Huneke and Roger Wiegand. Tensor products of modules, rigidity and local cohomology. \emph{Math. Scand}., 81(2):161--183, 1997. 

\bibitem{HWE}
Craig Huneke and Roger Wiegand. Correction to ``tensor products of modules and the rigidity of {T}or'', {M}ath. {A}nnalen, 299 (1994), 449--476. \emph{Mathematische Annalen}, 338(2):291--293, 2007.

\bibitem{Jo2} 
David A. Jorgensen. Tor and torsion on a complete intersection. \emph{J. Algebra}, 195(2):526--537, 1997.

\bibitem{Jor}
David A. Jorgensen. Vanishing of (co)homology over commutative rings. \emph{Comm. Algebra}, 29(5):1883--1898, 2001. 

\bibitem{Li}
Stephen Lichtenbaum. On the vanishing of Tor in regular local rings. \emph{Illinois J. Math.}, 10:220--226, 1966. 

\bibitem{Rotman}
Joseph J. Rotman. \emph{An introduction to homological algebra}, volume 85 of \emph{Pure and Applied Mathematics}. Academic Press Inc. [Harcourt Brace Jovanovich Publishers], New York, 1979. 

\bibitem{SS}
Arash Sadeghi. Symmetry in vanishing of Tate cohomology over Gorenstein rings. Preprint 2016; posted at arXiv:1608.04588.

\bibitem{Sa1}
Arash Sadeghi. Vanishing of cohomology over complete intersection rings. \emph{Glasg. Math. J.}, 57(2):445--455, 2015. 

\end{thebibliography}
\bibliographystyle{plain}


\end{document}